\documentclass[11pt]{amsart}

\usepackage{amsmath}
\usepackage{amscd}
\usepackage{amssymb}
\usepackage{boxedminipage}
\usepackage{enumerate}
\usepackage{graphics}
\usepackage{epic}
\usepackage{curves}
\usepackage{footmisc}

\begin{document}

\theoremstyle{plain} \newtheorem{theorem}{Theorem}[section]
\theoremstyle{plain} \newtheorem{lemma}[theorem]{Lemma}
\theoremstyle{plain} \newtheorem{proposition}[theorem]{Proposition}
\newtheorem{axioms}[theorem]{Axioms}
\newtheorem{corollary}[theorem]{Corollary}
\newtheorem{problem}[theorem]{Problem}
\newtheorem{conjecture}[theorem]{Conjecture}
\newtheorem{conjecture*}[]{Conjecture}
\newtheorem{matheorem}[theorem]{Main Theorem}
\newtheorem*{theoremidcol}{Theorem \ref{ResThmIdcol}}
\newtheorem*{theoremsyzcol}{Theorem \ref{ResThmSyzcol}}
\newtheorem*{theoremreg}{Theorem \ref{ColThmReg}}

\newcommand{\nr}{\refstepcounter{theorem}  
                   \noindent {\thetheorem .}}
\newcommand{\defi}{\medskip \noindent {\it Definition \nr} }
\newcommand{\defifin}{\medskip}
\newcommand{\eks}{\medskip \noindent {\it Example \nr} }
\newcommand{\eksfin}{\medskip}
\newcommand{\rem}{\medskip \noindent {\it Remark \nr} }
\newcommand{\remfin}{\medskip}
\newcommand{\spm}{\medskip \noindent {\it Question \nr} }
\newcommand{\spmfin}{\medskip}
\newcommand{\obs}{\medskip \noindent {\it Observation \nr} }
\newcommand{\obsfin}{\medskip}
\newcommand{\note}{\medskip \noindent {\it Notation \nr} }
\newcommand{\notefin}{\medskip}

\newcommand{\llabel}{\addtocounter{theorem}{-1}
\refstepcounter{theorem} \label}

\newcommand{\psp}[1]{{{\bf P}^{#1}}}
\newcommand{\psr}[1]{{\bf P}(#1)}
\newcommand{\op}{{\mathcal O}}
\newcommand{\opw}{\op_{\psr{W}}}
\newcommand{\go}{\op}

\newcommand{\ini}[1]{\text{in}(#1)}
\newcommand{\gin}[1]{\text{gin}(#1)}
\newcommand{\kk}{{\Bbbk}}
\newcommand{\pd}{\partial}
\newcommand{\vardel}{\partial}
\renewcommand{\tt}{{\bf t}}


\newcommand{\coh}{{{\text{{\rm coh}}}}}


\newcommand{\modv}[1]{{#1}\text{-{mod}}}
\newcommand{\modstab}[1]{{#1}-\underline{\text{mod}}}

\newcommand{\sut}{{}^{\tau}}
\newcommand{\sumit}{{}^{-\tau}}
\newcommand{\til}{\thicksim}

\newcommand{\totp}{\text{Tot}^{\prod}}
\newcommand{\dsum}{\bigoplus}
\newcommand{\dprod}{\prod}
\newcommand{\lsum}{\oplus}
\newcommand{\lprod}{\Pi}

\newcommand{\La}{{\Lambda}}

\newcommand{\sirstj}{\circledast}

\newcommand{\she}{\EuScript{S}\text{h}}
\newcommand{\cm}{\EuScript{CM}}
\newcommand{\cmd}{\EuScript{CM}^\dagger}
\newcommand{\cmri}{\EuScript{CM}^\circ}
\newcommand{\cler}{\EuScript{CL}}
\newcommand{\clerd}{\EuScript{CL}^\dagger}
\newcommand{\clerri}{\EuScript{CL}^\circ}
\newcommand{\gor}{\EuScript{G}}
\newcommand{\gF}{\mathcal{F}}
\newcommand{\gG}{\mathcal{G}}
\newcommand{\gM}{\mathcal{M}}
\newcommand{\gE}{\mathcal{E}}
\newcommand{\gD}{\mathcal{D}}
\newcommand{\gI}{\mathcal{I}}
\newcommand{\gP}{\mathcal{P}}
\newcommand{\gK}{\mathcal{K}}
\newcommand{\gL}{\mathcal{L}}
\newcommand{\gS}{\mathcal{S}}
\newcommand{\gC}{\mathcal{C}}
\newcommand{\gO}{\mathcal{O}}
\newcommand{\gJ}{\mathcal{J}}
\newcommand{\gU}{\mathcal{U}}
\newcommand{\mm}{\mathfrak{m}}

\newcommand{\dlim} {\varinjlim}
\newcommand{\ilim} {\varprojlim}

\newcommand{\CM}{\text{CM}}
\newcommand{\Mon}{\text{Mon}}


\newcommand{\Kom}{\text{Kom}}


\newcommand{\EH}{{\mathbf H}}
\newcommand{\res}{\text{res}}
\newcommand{\Hom}{\text{Hom}}
\newcommand{\inhom}{{\underline{\text{Hom}}}}
\newcommand{\Ext}{\text{Ext}}
\newcommand{\Tor}{\text{Tor}}
\newcommand{\ghom}{\mathcal{H}om}
\newcommand{\gext}{\mathcal{E}xt}
\newcommand{\id}{\text{{id}}}
\newcommand{\im}{\text{im}\,}
\newcommand{\codim} {\text{codim}\,}
\newcommand{\resol}{\text{resol}\,}
\newcommand{\rank}{\text{rank}\,}
\newcommand{\lpd}{\text{lpd}\,}
\newcommand{\coker}{\text{coker}\,}
\newcommand{\supp}{\text{supp}\,}
\newcommand{\Ad}{A_\cdot}
\newcommand{\Bd}{B_\cdot}
\newcommand{\Fd}{F_\cdot}
\newcommand{\Gd}{G_\cdot}


\newcommand{\sus}{\subseteq}
\newcommand{\sups}{\supseteq}
\newcommand{\pil}{\rightarrow}
\newcommand{\vpil}{\leftarrow}
\newcommand{\rpil}{\leftarrow}
\newcommand{\lpil}{\longrightarrow}
\newcommand{\inpil}{\hookrightarrow}
\newcommand{\pils}{\twoheadrightarrow}
\newcommand{\projpil}{\dashrightarrow}
\newcommand{\dotpil}{\dashrightarrow}
\newcommand{\adj}[2]{\overset{#1}{\underset{#2}{\rightleftarrows}}}
\newcommand{\mto}[1]{\stackrel{#1}\longrightarrow}
\newcommand{\vmto}[1]{\stackrel{#1}\longleftarrow}
\newcommand{\mtoelm}[1]{\stackrel{#1}\mapsto}

\newcommand{\eqv}{\Leftrightarrow}
\newcommand{\impl}{\Rightarrow}

\newcommand{\iso}{\cong}
\newcommand{\te}{\otimes}
\newcommand{\into}[1]{\hookrightarrow{#1}}
\newcommand{\ekv}{\Leftrightarrow}
\newcommand{\equi}{\simeq}
\newcommand{\isopil}{\overset{\cong}{\lpil}}
\newcommand{\equipil}{\overset{\equi}{\lpil}}
\newcommand{\ispil}{\isopil}
\newcommand{\vvi}{\langle}
\newcommand{\hvi}{\rangle}
\newcommand{\susneq}{\subsetneq}
\newcommand{\sgn}{\text{sign}}


\newcommand{\xd}{\check{x}}
\newcommand{\ortog}{\bot}
\newcommand{\tL}{\tilde{L}}
\newcommand{\tM}{\tilde{M}}
\newcommand{\tH}{\tilde{H}}
\newcommand{\tvH}{\widetilde{H}}
\newcommand{\tvh}{\widetilde{h}}
\newcommand{\tV}{\tilde{V}}
\newcommand{\tS}{\tilde{S}}
\newcommand{\tT}{\tilde{T}}
\newcommand{\tR}{\tilde{R}}
\newcommand{\tf}{\tilde{f}}
\newcommand{\ts}{\tilde{s}}
\newcommand{\tp}{\tilde{p}}
\newcommand{\tr}{\tilde{r}}
\newcommand{\tfst}{\tilde{f}_*}
\newcommand{\empt}{\emptyset}
\newcommand{\bfa}{{\bf a}}
\newcommand{\la}{\lambda}
\newcommand{\bfen}{{\mathbf 1}}
\newcommand{\ep}{e}
\newcommand{\Jc}{{\mathcal J}}

\newcommand{\ome}{\omega_E}

\newcommand{\bevis}{{\bf Proof. }}
\newcommand{\demofin}{\qed \vskip 3.5mm}
\newcommand{\nyp}[1]{\noindent {\bf (#1)}}
\newcommand{\demo}{{\it Proof. }}
\newcommand{\demodone}{\demofin}
\newcommand{\parg}{{\vskip 2mm \addtocounter{theorem}{1}  
                   \noindent {\bf \thetheorem .} \hskip 1.5mm }}

\newcommand{\red}{{\text{red}}}
\newcommand{\lcm}{{\text{lcm}}}
\newcommand{\pt}{{\scriptscriptstyle{\bullet}}}


\newcommand{\dl}{\Delta}
\newcommand{\cdel}{{C\Delta}}
\newcommand{\cdelp}{{C\Delta^{\prime}}}
\newcommand{\dlst}{\Delta^*}
\newcommand{\Sdl}{{\mathcal S}_{\dl}}
\newcommand{\lk}{\text{lk}}
\newcommand{\lkd}{\lk_\Delta}
\newcommand{\lkp}[2]{\lk_{#1} {#2}}
\newcommand{\del}{\Delta}
\newcommand{\delr}{\Delta_{-R}}
\newcommand{\dd}{{\dim \del}}

\renewcommand{\aa}{{\bf a}}
\newcommand{\bb}{{\bf b}}
\newcommand{\cc}{{\bf c}}
\newcommand{\xx}{{\bf x}}
\newcommand{\yy}{{\bf y}}
\newcommand{\zz}{{\bf z}}
\newcommand{\mv}{{\xx^{\aa_v}}}
\newcommand{\mF}{{\xx^{\aa_F}}}

\newcommand{\pnm}{{\bf P}^{n-1}}
\newcommand{\opnm}{{\go_{\pnm}}}
\newcommand{\ompnm}{\omega_{\pnm}}

\newcommand{\pn}{{\bf P}^n}
\newcommand{\hele}{{\bf Z}}

\newcommand{\dt}{{\displaystyle \cdot}}
\newcommand{\st}{\hskip 0.5mm {}^{\rule{0.4pt}{1.5mm}}}              
\newcommand{\disk}{\scriptscriptstyle{\bullet}}

\def\CC{{\mathbb C}}
\def\GG{{\mathbb G}}
\def\ZZ{{\mathbb Z}}
\def\NN{{\mathbb N}}
\def\RR{{\mathbb R}}
\def\OO{{\mathbb O}}
\def\QQ{{\mathbb Q}}
\def\VV{{\mathbb V}}
\def\PP{{\mathbb P}}
\def\EE{{\mathbb E}}
\def\FF{{\mathbb F}}
\def\AA{{\mathbb A}}

\title [Colorful Helly theorem and colorful resolutions of ideals]
{The colorful Helly theorem and colorful resolutions of ideals}
\author { Gunnar Fl{\o}ystad}
\address{ Dep. of Mathematics\\
          Johs. Brunsgt. 12\\
          5008 Bergen\\
          Norway}   
        
\email{ gunnar@mi.uib.no}

\begin{abstract}
We demonstrate that the topological Helly theorem and the algebraic
Auslander-Buchsbaum may be viewed as different versions of the same phenomenon. 
Using this correspondence we show how the colorful Helly theorem of I.Barany 
and its generalizations
by G.Kalai and R.Meshulam translates to the algebraic side. 
Our main results are algebraic generalizations
of these translations, which in particular give a syzygetic version
of Helly's theorem.

\end{abstract}

\maketitle

{\Small 2000 MSC : Primary 13D02. Secondary 13F55, 05E99.}

\section{Introduction}

The classical Helly theorem is one of the basic results of convex geometry.
It was subsequently generalized by Helly to a topological setting. In the last decade 
topological techniques has been infused into the study of resolutions of a monomial
ideal $I$ in a polynomial ring $S$ by the technique of cellular resolutions
\cite{MiSt}. 
Such resolutions are constructed
from a cell complex with monomial labellings on the cells satisfying certain 
topological conditions. We show that these topological conditions are
the same as the hypothesis of Helly's theorem, and that the conclusion of 
Helly's theorem corresponds to the weaker version of the Auslander-Buchsbaum
theorem one gets if one does not involve the concept of depth, but only that of 
dimension of the quotient ring  $S/I$.

The Helly theorem was generalized by I.B\'ar\'any \cite{Ba} to a colorful
version. By considering nerve complexes this has again recently been
abstracted and generalized by G.Kalai and R.Meshulam \cite{KaMe}. Their
result involves a pair $M \sus X$ wher $M$ is a matroidal complex and $X$ 
is a simplicial complex on a vertex set $V$. In the case when $M$ is 
a transversal matroid on a partition $V = V_1 \cup \cdots \cup V_{d+1}$
it corresponds to the colorful version of Barany. 
This latter result may be translated to the algebraic side in two ways. 
Either by the
notion of cellular resolutions, or by the Stanley-Reisner ring of the 
nerve complex.
In fact these two translations are related by Alexander duality.

The translation by the Stanley-Reisner ring of the nerve complex
admits algebraic generalizations to multigraded ideals. We consider a 
polynomial ring
$S$ whose set of  variables $X$ is partioned into a disjoint union 
$X = X_1 \cup \cdots \cup X_r$ and the variables in  $X_i$ have multidegree
$e_i$, the $i$'th coordinate vector. (One may think of the variables in $X_i$ as
having color $i$.) We investigate $\NN^r$-graded ideals $I$ of $S$ and their 
resolutions. The following is our first main result.
(The notion of $(d+1)$-regular ideal is explained in the beginning of Section 
\ref{ResSec}.)

\begin{theoremidcol} 
Let $S$ be a polynomial ring where the variables have
$r$ colors. Let $I$ be a $(d+1)$-regular ideal in $S$, homogeneous for 
the $\NN^r$-grading. Suppose $I$ has elements of pure color $i$ for each
$i = 1, \ldots, r$. Then for each color vector $(a_1, \ldots, a_r)$
where $\sum a_i = d+1$, there exists an element of $I$ with this color vector.
\end{theoremidcol}

This provides a generalization of the result of Kalai and Meshulam \cite{KaMe} in
the case of transversal matroids, since in our situation it applies to 
multigraded ideals and not
just to squarefree monomial ideals. Also the 
number of colors $r$ and the regularity $d+1$ may be arbitrary, and not the same
as in \cite{KaMe}. (But in the monomial case this generalization would 
have been a relatively
easy consequence of \cite{KaMe}.)
The theorem above is however a special case of our second main result
which is the following syzygetic version.

\begin{theoremsyzcol}
Let $S$ be a polynomial ring where the variables 
have $r$ colors. Let $I$ be an ideal in $S$ which is 
homogeneous for the $\NN^r$-grading, is generated in degree $d+1$ and 
is $d+1$-regular, so it has linear resolution. Suppose $I$ has elements of 
pure color $i$ for each $i = 1, \ldots, r$, and let $\Omega^l$
be the $l$'th syzygy module in an $\NN^r$-graded resolution of $S/I$
(so $\Omega^1 = I$). For each $l = 1, \ldots, r$ and each color vector 
$\aa = (a_1, \ldots, a_r)$ where $\sum a_i = d+l$ and $s$
the number of nonzero coordinates of $\aa$,
the vector space dimension of $(\Omega^l)_\aa$ is greater or equal to 
$\binom{s-1}{l-1}$.
\end{theoremsyzcol}

It is noteworthy that we do not prove Thorem \ref{ResThmIdcol}
directly and do not know how to do it. 
Rather we show Theorem \ref{ResThmSyzcol} by {\it descending} induction on the 
$\Omega^l$, starting from $\Omega^r$, and then Theorem \ref{ResThmIdcol} is 
easily deduced from the
special case of $\Omega^1$.

Letting $T = \kk[y_1, \ldots, y_r]$ be the polynomial ring in $r$ variables, 
our technique involves comparing the ideal $I \sus S$ and its natural image 
in $T$ by mapping the variables in $X_i$ to sufficiently general 
constant multiples of $y_i$. Along this vein we also 
prove the following.

\begin{theoremreg} 
Let $I \sus S$ be a multigraded ideal with elements of pure color $i$
for each $i$, and let $J \sus T$ be an ideal containing
the image of $I$ by the map $S \pil T$ (described above).
Then the regularity of $I$ is greater than or equal to the regularity of $J$.
\end{theoremreg}

While we study colored {\it homogeneous} ideals, 
until now they have mostly been studied in the monomial setting. 
One of the first is perhaps Stanley \cite{Sta} studying
balanced simplicial complexes. 
Then followed Bj\"orner, Frankl, Stanley \cite{BFS}
which classified flag $f$-vectors of ${\mathbf a}$-balanced 
Cohen-Macaulay simplicial complexes where ${\mathbf a} \in \NN^r$. 
A more recent paper is Babson, Novik \cite{BN} where they develop shifting
theory for colored homogeneous ideals, but rather focus on monomial ideals
to give another approach to \cite{BFS}. 
Nagel and Reiner
\cite{NR} give a very nice construction of the cellular resolution of 
colored monomial ideals associated to strongly stable ideals generated in a
single degree. 

\medskip
To give some more perspective on our approach, in studying graded ideals
in a polynomial ring much attention has been given to homogeneous (i.e. 
$\NN$-graded)
ideals and to monomial (i.e. $\NN^n$-graded) ideals. 
One way to approach intermediate cases is toric
ideals where the variables are attached degrees in some $\NN^r$.  
A toric ideal is uniquely determined by this association of degrees in $\NN^r$,
but
their class does not encompass the classes of monomial or homogeneous ideals.
Our approach is probably the simplest way of building a bridge encompassing
these two extremal 
cases: let the variables only have the unit vectors
in $\NN^r$ as degrees.

\medskip

The organization of the paper is as follows. In Section 2 we show that the
topological Helly theorem and the algebraic Auslander-Buchsbaum theorem are
basically different versions of the same phenomenon. In Section 3 we consider
the colorful Helly theorem and how it translates to the algebraic side,
either via cellular resolutions or via the nerve complex.
In Section 4 we state and prove our main results, Theorems \ref{ResThmIdcol},
\ref{ResThmSyzcol}, and \ref{ColThmReg}.

\medskip
\noindent {\bf Acknowledgements.} We thank B.Sturmfels for comments on the results
and in particular for providing the idea to Theorem \ref{ColThmReg}. 

\section{The topological Helly theorem and the Auslander-Buchsbaum theorem}

\subsection{The Helly theorem}
The classical Helly theorem is one of the founding theorems in convex geometry,
along with Radon's theorem and Carath\'eodory's theorem, and is the one which 
has been generalized in most directions.

\begin{theorem}[Helly 1908] Let $\{K_i\}_{i \in B}$ be a family of
convex subsets of $\RR^d$. If $\cap_{i \in B}K_i$ is empty, then there is  
$A \sus B$ of cardinality $d+1$ such that $\cap_{i \in A}K_i$ is empty.
\end{theorem}

Helly generalized this result in 1930 to a topological version.
We shall be interested in it in the context of polyhedral complexes.
Let $X$ be a bounded polyhedral complex. So $X$ is a finite family of 
convex polytopes with the following properties.

\begin{itemize}
\item If $P$ is in $X$, then all the faces of $P$ are in $X$.
\item If $P$ and $Q$ are in $X$, then $P \cap Q$ is a face in both
$P$ and $Q$.
\end{itemize}

Let  $V$ be the vertices of $X$, i.e. the $0$-dimensional faces.
Let $\{V_i\}_{i \in B}$ be a collection of subsets of $V$, and let $C_i$ be the 
subcomplex of $X$ induced from the vertex set $V_i$, consisting of the faces 
of $X$ whose vertices are all in $V_i$. Now if 
$X$ can be embedded in $\RR^d$, 
and for every $B^\prime \sus B$
the intersection $\cap_{i \in B^\prime} C_i$ is either empty or acyclic, 
then the topological realizations $|C_i|$ fulfill the conditions of the 
topological Helly theorem.

\begin{theorem}[Helly 1930] \label{TopThmTopHel} 
Let $\{K_i\}_{i \in B}$ be a family of closed
subsets in $\RR^d$ with empty interesection, 
such that for each $B^\prime \sus B$ the intersection
$\cap_{i \in B^\prime} K_i$ is either empty or acyclic over the field 
${\mathbb R}$. Then there
is $A \sus B$ of cardinality $\leq d+1$ such that $\cap_{i \in A} K_i$
is empty.
\end{theorem}

\begin{corollary}  \label{TopCorMin} Every minimal subset $A$ of $B$ such that 
$\cap_{i \in A} K_i$ is empty, has cardinality $\leq d+1$.
\end{corollary}


\rem Note that in Theorem \ref{TopThmTopHel} we only require the existence
of an $A$ with $\cap_{i \in A} K_i$ empty while in Corollary \ref{TopCorMin}
we say that every minimal $A$ with $\cap_{i \in A} K_i$ is empty has
cardinality $\leq d+1$. Given the hypothesis these two properties are
really equivalent.
\remfin

\subsection{Cellular resolutions} We now make a monomial labelling of the
vertices in $X$ by letting the vertex $v$ be labeled by
\begin{equation} \label{TopLabMon} \mv = \Pi_{v \not \in C_i} x_i. 
\end{equation}
In other words the variable $x_i$ is distributed to all vertices outside
$C_i$.

Before proceeding we recall some basic theory  of monomial labellings of cell
complexes and their associated cellular complexes of free modules
over a polynomial ring, following \cite[Ch.4]{MiSt}.

Let $\kk$ be a field and let $\tilde{\gC}_\cdot(X; \kk)$
be the reduced chain complex of $X$ over $\kk$. The term
$\tilde{\gC}_i(X;\kk)$ is the vector space $\oplus_{\dim F = i} \kk F$ 
with basis consisting of the $i$-dimensional polytopes of $X$, 
and differential
\[ \vardel_i(F) = \sum_{{\text{facets } G \sus F}} 
\sgn(G,F) \cdot G, \]
where $\sgn(G,F)$ is either $1$ or $-1$ and can be chosen consistently
making $\vardel$ a differential.
   
Now given a polynomial ring $S = \kk[x_i]_{i \in B}$ we may label 
each vertex $v$ of $X$ by a monomial $\mv$.
Each face $F$ is then labeled by $\mF$ which is the least common multiple 
$\lcm_{v \in F} \{ \mv\}$. Now we construct the {\it cellular complex} 
$\gF(X;\kk)$ consisting of free $S$-modules. The term $\gF_i(X ; \kk)$
is the free $S$-module $\oplus _{\dim F = i} SF$ with basis consisting of the
$i$-dimensional polytopes of $X$. The basis element $F$ is given degree
${\bf a}_F$. The differential is 
given by
\[  \vardel_i(F) = \sum_{\text{facets }F \sus G} 
\sgn(G,F) \frac{{\mathbf x}^{\aa_F}}{{\mathbf x}^{\aa_G}} \cdot G, \] 
which makes it homogeneous of degree $0$. 

We are interested in the case that $\gF(X;\kk)$ gives a free resolution
of $\coker \vardel_0 = S/I$ where $I$ is the monomial ideal generated by
the monomials $\mv$. For each $\bb \in \NN^r$, let $X_{\leq \bb}$ be the
subcomplex of $X$ induced on the vertices $v$ such that $\mv$ divides $\xx^\bb$.
This is the subcomplex consisting of all faces $F$ such that 
$\mF$ divides  $\xx^\bb$.  According to \cite[Ch.4]{MiSt}, $\gF(X;\kk)$ is a
free cellular resolution iff $X_{\leq \bb}$ is acyclic over $\kk$ or empty for 
every $\bb \in \NN^r$. Let $\ep_i$ be the $i$'th unit coordinate vector
in $\NN^r$ and let $\bfen = \sum_{i=1}^r \ep_i$.

The following provides an unexpected connection between Helly's theorem and 
resolutions of square free monomial ideals.

\begin{theorem} \label{TopThmHelRes}
Let $X$ be an acyclic polyhedral complex.
There is a one-to-one correspondence between the following.

a. Finite families $\{C_i\}_{i \in B}$ of induced subcomplexes of $X$ 
such that for each $B^\prime \sus B$, the intersection 
$\cap_{i \in B^\prime} C_i$ is either empty or acyclic.

b. Monomial labellings of $X$ with square free monomials in 
$\kk[x_i]_{i \in B}$ such that $\gF(X;\kk)$ gives a cellular resolution 
of the ideal generated by these monomials.

The correspondence is given as follows. 
Given a monomial labeling define $C_i, \,i \in B$, by letting 
$C_i = X_{\leq \bfen - \ep_i}$, 
and given a family $\{C_i\}$ define the monomials $\mv$, by 
$\mv = \Pi_{v \not \in C_i} x_i$.
\end{theorem}

\begin{proof}
Assume part a. To show part b. we must show that $X_{\leq \bb}$ is either
acyclic or empty for $\bb \in \NN^B$. It is enough to show this for 
$\bb \in \{0,1\}^B$ since we have a square free monomial labelling. 
If $\bb = \bfen$, then $X_{\leq \bb} = X$ which is acyclic. 
If $\bb < \bfen$ we have an intersection 
\[ X_{\leq \bb} = \cap_{\{i;b_i = 0\}} X_{\leq \bfen - \ep_i} \]
and $C_i = X_{\leq \bfen - \ep_i}$ by (\ref{TopLabMon}).

Now given part b., then $\cap_{i \in B^\prime} C_i$ is equal to 
$\cap_{i \in B^\prime} X_{\leq \bfen - \ep_i}$ which is $X_{\leq \bfen - 
\sum_{i \in B^\prime} \ep_i}$ and therefore is empty or acyclic.
\end{proof}

\subsection{The Auslander-Buchsbaum theorem}

We shall now translate the statements of the Helly theorem to statements
concerning the monomial ideal. First we have the following.

\begin{lemma} \label{TopLemIdeal} 
Let $A \sus B$. Then $\cap_{i \in A} C_i$ is empty 
if and only if the ideal $I$ is contained in $\langle x_i, i \in A \rangle$.
\end{lemma}

\begin{proof} 
Since $C_i = X_{\leq \bfen - \ep_i}$, that the intersection is empty means 
that for each monomial $\mv$ there exists $i \in A$ such that $x_i$
divides $\mv$. But this means that $\mv$ is an element of 
$\langle x_i; i \in A \rangle$ and so $I$ is included in this.
\end{proof}

Condition a. in Theorem \ref{TopThmHelRes} is the same as the
condition in the topological Helly theorem. Taking Corollary 
\ref{TopCorMin} and the lemma above
into consideration, the topological Helly theorem, Theorem 
\ref{TopThmTopHel}, translates to the following.

\begin{theorem} Suppose the monomial labelling (\ref{TopLabMon})
gives a cellular resolution of the ideal $I$ generated by these
monomials. Then every minimal prime ideal of $I$ has codimension 
$\leq d+1$. 
\end{theorem}

On the other hand we have the following classical result.

\begin{theorem}[Auslander-Buchsbaum 1955] Let $I$ be an ideal in 
the polynomial ring $S$. Then 
\[ \text{projective dimension}(S/I) + depth (S/I) = |B|.\]
\end{theorem}

Without the concept of depth but using only the concepts of dimension and
projective dimension this takes the following form.

\begin{corollary} \[ \codim I \leq \text{projective dimension} (S/I). \]
\end{corollary}

We see that in the situation we consider where $X$ gives a cellular resolution
of $I$, the projective dimension of $I$ is less or equal to $d+1$, 
with equality if the resolution is minimal.
We thus see that the Auslander-Buchsbaum theorem and Helly's theorem are
simply different versions of the same phenomenon.
Note however that not every monomial ideal has a cellular resolution
giving a {\it minimal} free resolution of the ideal, see
\cite{Ve}.
\subsection{The nerve complex}

Recall that a simplicial complex $\Delta$ on a subset $B$ is a 
family of subsets of $B$ such that if $F \in \Delta$ and $G \sus F$
then $G \in \Delta$. Simplicial complexes on $B$ are in one-to-one 
correspondence
with square free monomial ideals in $\kk[x_i]_{i \in B}$, where $\Delta$
corresponds to the ideal $I_{\Delta}$ generated by $\xx^\tau$ where
$\tau$ ranges over the nonfaces of $\Delta$.
 If $R \sus B$ the restriction $\Delta_R$ consist
of all faces $F$ in $\Delta$ which are contained in $R$. 
The simplicial complex is called {\it $d$-Leray} if 
the reduced homology groups $\tH_i(\Delta_R, \kk)$ vanish whenever $R \sus B$ 
and $i \geq d$.

The {\it nerve complex} of the family $\{C_i\}_{i \in B}$ is the simplicial 
complex consisting of the subsets $A \sus B$ such that $\cap_{i \in A} C_i$ is
nonempty. 
The following are standard facts. 

\begin{proposition} \label{TopProNerve} 
Let $\{C_i\}_{i \in B}$ be a family of induced subcomplexes of a 
polyhedral complex $X$, with nerve complex $N$.
Suppose that any intersection of elements in this family is empty or acyclic.

a. The union $\cup_{i \in B} C_i$ is homotopy 
equivalent to $N$. 

b.  If $X$ has dimension $d$ and $\tH_d(X, \kk) = 0$,
    the nerve complex $N$ is $d$-Leray.

\end{proposition}

\begin{proof} Part a. follows for instance from \cite[Theorem 10.7]{Bj}. 
Part b. follows because for $R \sus B$, the restriction $N_R$ is the nerve
complex of $\{C_i \}_{i \in R}$, and so $N_R$ is homotopy equivalent
to $\cup_{i \in R} C_i$ which is a subcomplex of $X$. Hence
$\tH_i( \cup_{i \in R} C_i)$ must vanish for $i \geq d$. 
\end{proof}

The {\it Alexander dual} simplicial complex of $\Delta$ consist of those
subsets $F$ of $B$ such that their complements, $F^c$, in $B$ are not elements of
$\Delta$.

\begin{proposition} Given any finite family $\{C_i\}_{i \in B}$ 
of induced subcomplexes of $X$, let $I_\Delta$ be generated by the monomials
in (\ref{TopLabMon}).
The nerve complex $N$ of the family $\{C_i\}_{i \in B}$ is the Alexander dual 
of the simplicial complex $\Delta$.
\end{proposition}

\begin{proof}
We must show that a square free monomial $\xx^A$ is in $I_\Delta$ 
iff $A^c$ is in $N$.
Given a generator $\mv$, 
we have $v \in C_i$ for every $i \not \in \supp \aa_v$ (the support is the 
set of positions of nonzero coordinates). Hence
$v \in \cap_{i \in (\supp\aa_v)^c} C_i$ so this intersection is not empty and hence
$(\supp \aa_v)^c$ is in the nerve complex.

Conversely, if $v$ is in $\cap_{i \in A^c} C_i$, then $\aa_v \leq A$
so $\xx^A$ is in $I_\Delta$.
\end{proof}

\begin{corollary}
If $I_\Delta$ is a Cohen-Macaulay monomial ideal of codimension $d+1$, 
then $I_N$ has $(d+1)$-linear resolution.
\end{corollary}

\begin{proof} This follows from \cite{ER}, since $N$ is the Alexander dual
of $\Delta$.
\end{proof} 

\section{The colorful Helly theorem} \label{ColSec}

Helly's theorem was generalized by I.Barany \cite{Ba} to the so called
colorful Helly theorem. This was again generalized by G.Kalai and R.
Meshulam \cite{KaMe}. The following version of the colorful Helly
theorem is a specialization of their result, which we recall in Theorem
\ref{ResThmKM}. It follows from Corollary \ref{ResCorTransv} 
by letting $Y$ there be the nerve complex of the
$C_i$ and letting the $V_p$ be $B_p$. 

\begin{theorem} \label{ColThmKalai} 
Let $X$ be a polyhedral complex of dimension $d$ with
$\tH_d(X;\kk) = 0$.
Let $\{C_i\}_{i \in B_p}$ for $p = 1, \ldots, d+1$ be $d+1$ finite families
of induced subcomplexes of $X$, such that for every $A \sus \cup_p B_p$
the intersection $\cap_{i \in A} C_i$ is empty or acyclic. If every 
$\cap_{i \in B_p} C_i$ is empty, there exists $i_p \in B_p$ for each
$p = 1, \ldots, d+1$ such that $\cap_{p=1}^{d+1} C_{i_p}$ is empty.
\end{theorem}

Note that we think of the index sets as disjoint
although $C_i$ may be equal to $C_j$ for $i$ and $j$ in distinct
index sets, or even in the same index set.

\rem The topological Helly theorem follows from the colorful theorem
above by letting all the families $\{C_i\}_{i \in B_p}$ be equal.
\remfin

 Let $B$ be the  disjoint union of the $B_p$ and let $S = k[x_i]_{i \in B}$.
(One may think of the variables $x_i$ for $i \in B_p$ as having a given
color $p$.)
Let $I_\Delta \sus S$ be the associated Stanley-Reisner ideal of the
family $\{C_i \}_{i \in B}$, 
given by the correspondence in Theorem \ref{TopThmHelRes}.  
The following is then equivalent to the theorem above.

\begin{theorem} Suppose $I_\Delta$ is contained in each ideal 
$\langle x_i; i \in B_p \rangle$ generated by variables of the same 
color $p$, for $p = 1, \ldots, d+1$. Then there are $i_p \in B_p$
for each $p$ such that $I_\Delta$ is contained in the ideal $\langle x_{i_1}, 
\ldots, x_{i_{d+1}}\rangle$ generated by variables of each color.
\end{theorem}

\begin{proof}
Immediate from the theorem above when we take into consideration 
Lemma \ref{TopLemIdeal} and Theorem \ref{TopThmHelRes}.
\end{proof}

If $A \sus B$ let $a_p$ be the cardinality $|A \cap B_p|$. We say that
the ideal 
$\langle x_i; i \in A \rangle$ has color vector $(a_1, \ldots, a_{d+1})$.

\begin{corollary} Suppose $I_\Delta$ is a Cohen-Macaulay ideal of 
codimension $d+1$. If $I_\Delta$ has associated prime ideals of pure color
$(d+1)\cdot \ep_i$ for each $i = 1, \ldots, d+1$, then it has an associated
prime ideal with color vector $\sum_{i = 1}^{d+1} \ep_i$.
\end{corollary}

\begin{proof} In this case all associated prime ideals are generated by
$d+1$ variables.
\end{proof}

The following
is an equivalent formulation of Theorem \ref{ColThmKalai} and is the
form which will inspire the results in the next Section \ref{ResSec}. 
Say that a square free monomial has color vector $(a_1, \ldots, a_{d+1})$
if it contains $a_i$ variables of color $i$.

\begin{theorem} \label{ColThmNerve} Let $N$ be the nerve complex of the family
$\{C_i\}_{i \in B}$ of Theorem \ref{ColThmKalai}.
If $I_N$ has monomials of pure color $p$ for $p = 1, \ldots, d+1$,
then $I_N$ has a monomial of color $\sum_{i=1}^{d+1} \ep_i$.
\end{theorem}

\begin{proof} The intersection $\cap_{i \in A} C_i$ is empty iff
$\Pi_{i \in A} x_i$ is a monomial in $I_N$.
\end{proof}

\section{Colorful resolutions of ideals} \label{ResSec}

\subsection{Main results}
Let $X = X_1 \cup X_2 \cup \cdots \cup X_r$ be a partition of 
the variables in a polynomial ring $S$. Letting the variables in $X_i$
have multidegree $\ep_i \in \NN^r$, the $i$'th coordinate vector, 
we get an $\NN^r$-grading of the polynomial ring. A homogeneous polynomial
for this grading with multidegree $\aa = (a_1, \ldots, a_r)$ will be
said to have color vector $(a_1, \ldots, a_r)$. A polynomial with 
color vector $r_i \ep_i$ where $r_i$ is a positive integer is said to
have pure color $i$.

If $I$ is an ideal in $S$, homogeneous for this grading and containing polynomials
of pure color $i$ for $i = 1, \ldots, r$, then it may be a complete intersection 
of these, in which case $I$ does not have more generators. However
if we put conditions on the regularity of $I$ this changes. 
We recall this notion. Let $F_\pt$ be a minimal free resolutions of $I$
with terms $F_p = \oplus_{i \in \hele} S(-i)^{\beta_{p,i}}$. We say $I$ is
$m$-regular if $i \leq m+p$ for every nonzero $\beta_{p,i}$. 
This may be shown to be equivalent to the truncated ideal $\oplus_{p \geq m}
I_p$ having linear resolution. For a simplicial complex $\Delta$ it follows by
the description of the Betti numbers of $I_{\Delta}$ by reduced homology
groups, see \cite{MiSt}, Corollar 5.12, that 
$I_\Delta$ is $(d+1)$-regular if and only if $\Delta$ is $d$-Leray.
The latter hypothesis is fulfilled in Theorem \ref{ColThmNerve} for the nerve 
complex $N$, and the conclusion is that $I_N$ contains an element with 
color vector 
$\sum_{i=1}^{d+1} \ep_i$. The following generalizes this.

\begin{theorem} \label{ResThmIdcol} 
Let $S$ be a polynomial ring where the variables have
$r$ colors. Let $I$ be a $(d+1)$-regular ideal in $S$, homogeneous for 
the $\NN^r$-grading. Suppose $I$ has elements of pure color $i$ for each
$i = 1, \ldots, r$. Then for each color vector $(a_1, \ldots, a_r)$
where $\sum a_i = d+1$, there exists an element of $I$ with this color vector.
\end{theorem}

\eks Let there be $r=3$ colors and suppose $I$ contains polynomials of 
multidegree $(3,0,0)$, $(0,4,0)$, and $(0,0,5)$. If $I$ is generated by
them, it is a complete intersection and the regularity is $3+4+5-2 = 10$.
Suppose then that the regularity is $9,8,7$, and so on. What can be said
of the generators of $I$? 

If the regularity is $9$, by the above theorem
it must contain a generator of multidegree $\leq (2,3,4)$ (for the partial order
where $\aa \leq \bb$ if in each coordinate $a_i \leq b_i$).

If the regularity is $8$, it must by the above contain a generator of multidegree
$\leq (1,3,4)$, a generator of multidegree $\leq (2,2,4)$, and a generator 
of multidegree $\leq (2,3,3)$. In the same way we may continue and get 
requirements on the generators of $I$ if its regularity is $7,6$, and $5$ also.
\eksfin

The most obvious example of an ideal as in the theorem above is of course the power
$\mm^{d+1} = (x_1, \ldots, x_r)^{d+1}$ in $\kk[x_1, \ldots, x_r]$. Let 
$\aa = (a_1, \ldots, a_r)$ be a multidegree with $\sum a_i = d+j$ and 
support $s = \text{supp}(\aa)$ defined as the number of coordinates 
$a_i$ which are nonzero. Then the multigraded Betti number 
$\beta_{j, \aa}(S/\mm^{d+1})$ is equal to $\binom{s-1}{j-1}$. Motivated
by this we have a syzygetic version of the colorful
Helly theorem.

\begin{theorem} \label{ResThmSyzcol}
Let $S$ be a polynomial ring where the variables 
have $r$ colors. Let $I$ be an ideal in $S$ which is 
homogeneous for the $\NN^r$-grading, is generated in degree $d+1$ and 
is $(d+1)$-regular, so it has linear resolution. Suppose $I$ has elements of 
pure color $i$ for each $i = 1, \ldots, r$, and let $\Omega^l$
be the $l$'th syzygy module in an $\NN^r$-graded resolution of $S/I$
(so $\Omega^1 = I$). For each $l = 1, \ldots, r$ and each color vector 
$\aa = (a_1, \ldots, a_r)$ where $\sum a_i = d+l$ and $s = \text{supp}(\aa)$,
the vector space dimension of $(\Omega^l)_\aa$ is greater than or equal to 
$\binom{s-1}{l-1}$.
\end{theorem}

Theorem \ref{ResThmIdcol} may be deduced from this. Simply apply 
Theorem \ref{ResThmSyzcol} to the truncated ideal $I_{\geq d+1} = 
\oplus_{p \geq d+1} I_p$. Then Theorem \ref{ResThmIdcol} is the special case
$l=1$. Our goal is therefore now to prove Theorem \ref{ResThmSyzcol}.

\subsection{The result of Kalai and Meshulam}
Let $M$ be a matroid on a finite set $V$
(see Oxley \cite{Ox} or, relating it more directly to simplicial complexes,
Stanley \cite[III.3]{St}). 
This gives
rise to a simplicial complex consisting of the independent sets of the 
matroid. If $\rho$ is the rank function of the matroid, this simplicial
complex consist of all $S \sus V$ such that $\rho(S) = |S|$. Kalai
and Meshulam \cite[Thm. 1.6]{KaMe} show the following.

\begin{theorem}[Kalai, Meshulam 2004] \label{ResThmKM}
Let $Y$ be a $d$-Leray complex on $V$ 
and $M$ a matroid complex on $V$ such that $M \sus Y$. Then there is a 
simplex $\tau \in Y$ such that $\rho(V - \tau) \leq d$.
\end{theorem}

In case we have a partition $V = V_1 \cup \cdots \cup V_{d+1}$ we get the following.

\begin{corollary} \label{ResCorTransv} Let $M$ be the transversal 
matroid on the sets $V_1, \ldots, V_{d+1}$, i.e. the bases consist of all
$S \sus V$ such that the cardinality of each $S \cap V_i$ is one, 
and let $Y$ be a $d$-Leray complex containing $M$.
Then there is a simplex $\tau$ such that $(V - \tau) \cap V_i$ 
is empty for some $i$, or in other words $\tau \sups V_i$. 
\end{corollary}

In the monomial case it is not difficult, by polarizing, to prove
Theorem \ref{ResThmIdcol} as a consequence of Corollary \ref{ResCorTransv}.
However our result is a simultaneous generalization of this fact to arbitrary multigraded
ideals, to two parameters, $r$ the number of colors, and $d+1$ the regularity, 
and to higher syzygies of the ideal.

It is particularly worth noting that we do not 
prove Theorem \ref{ResThmIdcol} in any direct way, and we do not know how to
do it. Rather, it comes out as a special case of Theorem \ref{ResThmSyzcol},
which is proved by showing that it holds for the $\Omega^l$
by {\it descending} induction on $l$, starting from $\Omega^r$.

\subsection{Comparing the resolution of $I$ to resolutions of monomial ideals}

Let $T = \kk[y_1, \ldots, y_r]$. To start with we will look at 
monomial ideals $J$ in $T$ such that $T/J$ is artinian, i.e. $J$
contains an ideal $K = (y_1^{a_1}, \ldots, y_r^{a_r})$ generated by 
powers of variables. There is then a surjection $T/K \mto{p} T/J$ which
lifts to a map of minimal resolutions $\Bd \mto{\tilde{p}} \Ad$. 

\begin{lemma} \label{ColLemArt} 
Let $Te$ be the last term in the minimal resolution
$\Bd$ of $T/K$ (so $e$ has multidegree  $(a_1, \ldots, a_r)$),
and let $A_r = \oplus_1^n Te_i$ be the last term in the minimal free
resolution $\Ad$ of $T/J$. Suppose in the lifting $\Bd \mto{\tilde{p}}
\Ad$ of $T/K \pil T/J$ that $e \mapsto \sum m_i e_i$. Then
each $m_i \neq 0$. 
\end{lemma}

\begin{proof} Let $\omega_T \iso S(-\bfen)$ be the canonical module of $T$.
Dualizing the map $\tilde{p}$ we get a map
\[ \Hom_T(\Ad, \omega_T) \mto{\Hom_T(\tilde{p}, \omega_T)}
\Hom_T(\Bd, \omega_T). \]
Since $T/K$ and $T/J$ are artinian, this map will be a lifting
of the map
\[ \Ext^r_T(T/J, \omega_T) \pil \Ext^r_T(T/K, \omega_T) \]
 to their minimal free resolutions. But this map is simply
\begin{equation} \label{ColLigTJK} (T/J)^* \pil (T/K)^*,
\end{equation}
where $()^*$ denotes the vector space dual $\Hom_\kk(-,\kk)$, 
and so this map is injective.

Now $\Hom_T(A_r, \omega_T)$ is $\oplus_1^n Tu_i$ where the $u_i$ are
a dual basis of the $e_i$, and correspondingly let $u$ be the dual 
basis element of $e$. Then we will have
\[ u_i \mtoelm{\Hom_T(\tilde{p}, \omega_T)} m_i u. \]
Each $u_i$ maps to a minimal generator of $(T/J)^*$. If $m_i$ was $0$, 
this generator would map to $0$ in $(T/K)^*$, but since (\ref{ColLigTJK}) 
is injective this does not happen.
\end{proof}

Let $\lambda : X \pil \kk$ be a function associating to each variable
in $X$ a consant in $\kk$. 
There is then a map 
\begin{equation}
\label{ColLigPl} p_\lambda : S = \kk[X] \pil \kk[y_1, \ldots, y_r] = T
\end{equation}
sending each element $x$ in $X_i$ to $\lambda(x)y_i$. 
Given a multigraded ideal $I$ in $S$ 
with elements of pure color $i$ for each $i$, we can compare its 
regularity to the regularity of ideals in $T$.

\begin{theorem} \label{ColThmReg}
Let $I \sus S$ be a multigraded ideal with elements of pure color $i$
for each $i$.
The image $p_\lambda(I)$ by the map (\ref{ColLigPl}) is then a monomial
ideal.  Let $\lambda$ be sufficiently general so that for each $i$
some element of pure color $i$ in $I$ has nonzero image.
If $J \sus T$ is an ideal containing
the image $p_\lambda(I)$, its regularity is less than or equal to the 
regularity of $I$.
\end{theorem}

\begin{proof}
Let $P_i$ be elements of pure color $i$ in $I$ for $i = 1, \ldots, r$, such
that $P_i$ maps to a nonzero multiple of $y_i^{a_i}$. Let $H$ be the ideal
$(P_1, \ldots, P_r)$. We get a commutative diagram
\[ \begin{CD}
S/H @>>> T/K \\
@VVV @VVV \\
S/I @>>> T/J.
\end{CD} \] 
In the minimal free resolution $\Ad$ of $T/J$ let $A_r = \oplus_1^m Te_i$.
Since $T/J $ is artinian, its regularity is $\max \{ \deg(e_i) - r \}$. 
Let $\Fd$ be the minimal free resolution of $S/I$ and let 
$F_r = \oplus_1^{m^\prime} Sf_i$. The maps in the diagram above lift
to maps of free resolutions. In homological degree $r$ of the resolution
of $S/H$ we have a free $S$-module of rank one. If we consider its image
in $T/J$ and use Lemma \ref{ColLemArt} together with the commutativity
of the diagram, we see that by the map $F_r \mto{\tilde{p}_r} A_r$
there must for every $e_i$ exist an $f_j$ such that the composition
\[ Sf_j \pil \Fd \pil \Ad \pil Te_i \]
is nonzero. But then 
\[ \max \{ \deg(f_j) - r \} \geq \max \{ \deg(e_i) - r\} \]
so we get our statement.
\end{proof}

\begin{corollary} If $I$ has linear resolution, then $p_\lambda(I)$ 
has linear resolution (with the generality assumption on $\lambda)$.
\end{corollary}

\spm Is the corollary above true without the hypothesis on $I$ that
it contains elements of pure color $i$ for each $i$?
\spmfin

Now we shall proceed to prove Theorem \ref{ResThmSyzcol} and thereby
also the special case Theorem \ref{ResThmIdcol}. 

\begin{proof}[Proof of Theorem \ref{ResThmSyzcol}]
Let $J = \mm^{d+1}$ be the power of the maximal ideal in $T$. 
There is then a completely explicit form of the minimal
free resolution $\Ad$ of $\mm^{d+1}$, the Eliahou-Kervaire resolution, 
\cite{EK} or \cite{PeSt}, which we now describe in this particular case. 
For a monomial
$m$ let $\max(m) = \max \{ q \, | \, y_q \text{ divides } m \}$. 
Let $A_p$ be the free $T$-module with basis elements $(m; j_1, \ldots, j_{p-1})$
where $m$ is a monomial of degree $d+1$ and 
$1 \leq j_1 < \cdots < j_{p-1} < \max(m)$. 
These basis elements are considered to have degree $d+p$.
The differential of $\Ad$ is given by sending a basis element
$(my_k ; j_1, \ldots, j_{p-1})$ with  $\max(my_k) = k$ to 
\begin{eqnarray} \label{ColLigDiff} 
& & \sum_q (-1)^q y_{j_q} (my_k; j_1, \ldots, \hat{j_q}, \ldots, j_{p-1}) \\
& - & \sum_q (-1)^q y_k (my_{j_q};  j_1, \ldots, \hat{j_q}, \ldots, j_{p-1}).
\notag
\end{eqnarray}
(If a term in the second sum has $\max(my_{j_q}) \geq j_{p-1}$, the term
is considered to be zero.)
In homological degree $r$ we know by the proof of Theorem \ref{ColThmReg}
that for each basis element $e = (my_r; 1,2,\ldots, r-1)$ of $A_r$ there
is a basis element $f$ of $F_r$ such that the composition
\[ Sf \pil F_r \pil A_r \pil Te \]
maps $f$  to $ne$ where $n$ is a nonzero monomial. 
Since the $f$ and $e$ have the same total degree (the resolution is linear), 
we may assume that the
monomial $n$ is $1$. Since the map $F_r \mto{\tilde{p}_r} A_r$ 
is multihomogeneous and every two of  the basis elements of $A_r$ have 
different multidegrees, we will in fact have $f \mtoelm{\tilde{p}_r} e$.
This shows that the theorem holds when $l = r$.

We will now show Theorem \ref{ResThmSyzcol} 
by descending induction on $l$. 
Let $(my_k; \Jc^\prime)$ be a given basis element of $A_l$ where $|\Jc^\prime|= l-1$ and 
$\max(my_k) = k > \max(\Jc^\prime)$.

\medskip 

1. Suppose $k > l$. Then 
there is an $r < k$
not in $\Jc^\prime$, and let $\Jc = \Jc^\prime \cup \{r\}$.
Consider now the image of  $(my_k;\Jc)$ given by (\ref{ColLigDiff}).
No other basis element involved in this image has the same multidegree
as $(my_k; \Jc^\prime)$. By induction we assume 
there is an element $f$ in $F_{l+1}$, that forms part of a basis of $F_{l+1}$, such 
that $f \mtoelm{\tilde{p}_{l+1}} (my_k; \Jc)$. The differential in $\Fd$
maps $f$ to $\sum_1^{|X|} x_i f_i$ where each $f_i$ may be considered a
basis element of $F_l$ if nonzero. 
Since the map $\tilde{p}_l$ is homogeneous, a scalar multiple of some $f_i$ 
must map to $(my_k; \Jc^\prime)$. 

\medskip

2. Suppose $k = l$. 
Then $\Jc^\prime = \{1,2, \ldots, k-1 \}$ and let $\Jc = \Jc^\prime \cup \{k\}$. 
The image of 
$(my_{k+1};\Jc)$ by the differential in $\Ad$ will be
\begin{equation} \label{ColLigDiff2} 
\sum_{q = 1}^k(-1)^q y_q(my_{k+1}; \Jc \backslash \{ q \}) \\
- \sum_{q=1}^k (-1)^q y_{k+1} (my_q;\Jc \backslash \{q\}) .
\end{equation}
Since $\max(my_q) \leq k$ and $\max \Jc \backslash \{q\} = k$ when $q \neq k$, 
all the terms in the second sum become zero save
\[ (-1)^k y_{k+1} (my_k; \Jc^\prime) \]
and this term is the only one in (\ref{ColLigDiff2}) involving a
basis element with the multidegree of $(my_k; \Jc^\prime)$. 
By the same argument as in 1. above, there is a basis element $f$
in $F_l$  that maps to $(my_k;\Jc^\prime)$. 
This concludes the proof of the theorem.
\end{proof}

The following is a consequence of Theorem \ref{ResThmIdcol}.

\begin{corollary} Let $J$ in $\kk[y_1, \ldots, y_r]$ be a monomial ideal with linear
resolution generated in degree $d+1$. If $J$ contains the pure powers $y_i^{d+1}$ 
for each $i$, then $J$ is $(y_1, \ldots, y_r)^{d+1}$.
\end{corollary}

This also follows from the fact that the vector space dimensions of the graded pieces of the
socle of $S/J$ is determined by the last
term of the resolution of $S/J$. The resolution being linear means that the socle of 
$S/J$ is concentrated in degree one lower than the generators of $J$.  

Let $\Delta(d+1)$ be the geometric simplex in ${\mathbb R}^r$ 
defined as the convex hull of all $r$-tuples $(a_1, \ldots, a_r)$
of non-negative integers with $\sum a_i = d+1$. The monomials of degree $d+1$ of a
monomial ideal in $\kk[y_1, \ldots, y_r]$ may be identified with a subset of 
$\Delta(d+1)$. Under 
the hypothesis of $d+1$-linearity the corollary above 
says that if the ideal contains the extreme points of $\Delta(d+1)$ then it
contains all of its lattice points. This suggests the following more general problem.

\begin{problem}
Let $J$ in $\kk[y_1, \ldots, y_r]$ be a monomial ideal with linear resolution
generated in degree $d+1$. What can be said of the ``topology'' of the generating monomials
of $J$ considered as elements of $\Delta(d+1)$?
\end{problem}

\end{document}